\documentclass[11pt]{amsart}
\usepackage{amsfonts}
\usepackage{amsmath}
\usepackage{amssymb,amsthm,amscd}
\usepackage[mathscr]{eucal}
\usepackage{graphicx}
\usepackage{float}
\usepackage{color}
\usepackage[all]{xy}
\newtheorem{theorem}{Theorem}[section]
\newtheorem{lemma}{Lemma}[section]

\theoremstyle{remark}

\title[A semilinear partial differential equation induced by Hermitian Yang-Mills metrics]{ A semilinear partial differential equation induced by Hermitian Yang-Mills metrics
}

\author{Yuxuan Li}
\author{Wubin Zhou}

\address{School of Mathematical Science\\ Tongji University \\ Shanghai
200092, China} \email{1653454@tongji.edu.cn}

\address{  School of Mathematical Sciences\\ Tongji University \\
Shanghai 200092, China} \email{wbzhou@tongji.edu.cn}

\begin{document}
 \begin{abstract}
This paper will discuss a class of semilinear partial differential equations induced by  studying the limiting behaviour of Hermitian Yang-Mills metrics.   We will study the radial symmetry of the $C^{2}$ global solution of this equation in $\mathbb{R}^{2}$ and the existence of $C^{2,\alpha}$ solution of the Dirichlet boundary value problem in any bounded domain.\\
\textsc{Keywords.} Hermitian Yang-Mills metric, $C^{k}$-estimates, Boundary value problems\\
\textsc{Msc(2010).} 53C07, 32Q20
\end{abstract}
\maketitle

\section{Introduction}

Let $X$ be a K\"ahler manifold with a family of K\"ahler metrics $\omega_\varepsilon$, and let $V$ be a slope stable holomorphic vector bundle over $X$. According to the Donaldson-Uhlenbeck-Yau theorem \cite{UY},  $V$ admits unique fully irreducible Hermitian-Yang-Mills metrics $H_\varepsilon$ associated to each $\omega_\varepsilon$.  Similar to study the limiting behaviour Ricci flat metrics, Professor Jixiang Fu \cite{Fu}  studied the limiting behaviour of  Hermitian Yang-Mills metrics $H_\varepsilon$ when $\omega_\varepsilon$ goes to a large K\"ahler metric limit. A critical step in \cite{Fu}  is  to  explicitly construct a family of Hermitian-Yang-Mills metrics  by solving the following semilinear partial differential equation on unit ball $B_1(0)$ of $\mathbb{R}^2$
\begin{equation}\label{fu1}
\left\{
\begin{aligned}
\Delta u &=\varepsilon^{-2}\left(e^{u}-\left(x^{2}+y^{2}\right) e^{-u}\right) &\text { in } B_1(0),\\
u&=0  &\text { on } \partial B_1(0).
\end{aligned}
\right.
\end{equation}
Here $\varepsilon$ is a constant and $(x,y)$ is the coordinate of $\mathbb{R}^2$.  By the symmetry  of  the domain $B_1(0)$ and using  reference \cite{GNN},  Jixiang Fu proved the equation (\ref{fu1}) has a unique radially symmetric solution.  However, this method can not  be applied to non-symmetric domain $\Omega$ in $\mathbb{R}^2$.

In this paper, we first study  the following equation defined a bounded connected domain $\Omega\subset \mathbb{R}^2$ with Dirichlet boundary value

\begin{equation}\label{fu2}
\left\{
\begin{aligned}
\Delta u &=\varepsilon^{-2}\left(e^{u}-\left(x^{2}+y^{2}\right) e^{-u}\right) &\text { in } \Omega,\\
u&=g  &\text { on } \partial \Omega.
\end{aligned}
\right.
\end{equation}
For existence and unique of the solution of (\ref{fu2}), we have the following theorem
\begin{theorem}\label{ThmE}
If $\partial\Omega$ is  $C^{2,\alpha}$ and  $g\in C^{2,\alpha}(\partial\Omega)$,  there is a unique solution $u\in C^{2,\alpha}(\Omega)$ to equation (\ref{fu2}) . Especially if $\partial\Omega$ and $g$ is smooth,  the solution $u$ is smooth.
\end{theorem}
This theorem will give a Hermitian Yang-Mills metrics on a certain K\"ahler manifold given by \cite{Fu}. On the other hand,  the equation (\ref{fu1}) can be defined on whole space $\mathbb{R}^2$. It is natural to explore whether  the global solution of (\ref{fu1}) is radially symmetric. The symmetry of global solutions of some semilinear equations has been investigated in \cite{C} and \cite{GNN} under the assumption $u(x,y)$ decays  to zero at a certain rate as $r^2=x^2+y^2\to +\infty$. But they do not fit the equation(\ref{fu1}) since one can see the global solution $u$  is not bounded. Similar to  \cite{ C,GNN}, by using moving plane method and maximum principle, we get the following theorem

\begin{theorem}\label{ThmS}
For any given constant $c$,  if  the global $C^{2}$ solution $u$ of
\begin{equation}\label{fu3}
\Delta u =\varepsilon^{-2}\left(e^{u}-\left(x^{2}+y^{2}\right) e^{-u}\right) \qquad \text { in } \mathbb{R}^2
\end{equation}
satisfies $$ u(s)-u(t) \rightarrow 0 \quad\text{ as}\quad |s|,|t| \rightarrow \infty\quad \text{and}\quad |s|-|t|=c,$$ then $u$ is radially symmetric and $\frac{\partial u}{\partial r} \geqslant0$. Here $s,t \in \mathbb{R}^2$.
\end{theorem}

One may observe $\frac{1}{2}\log(x^2+y^2)$ is a singular solution to the equation (\ref{fu3}) and also satisfies $\log (|s|)-\log(|t|)\to 0$ as $|t|-|s|=c$ and $|s|, |t|\to \infty$, so the assumption of Theorem \ref{ThmS} is natural and reasonable.

The next part of this paper  will give the detailed  proof of Theorem \ref{ThmE} and Theorem \ref{ThmS}.

\textbf{Acknowledgements}. The authors would like to thank Professor Jixiang Fu for telling us this equation (\ref{fu1}) and some usefully discussions.  Zhou is supported by the National Natural Science Foundation of China (Grant No.11701426).

\section{Existence of solution of the Dirichlet boundary value problem}
In this section we will prove Theorem \ref{ThmE}.  One can use Chapter 14 in \cite{Taylor} to show the existence of the equation \ref{fu2}  by using  variational method.  Here we take Leray-Schauder existence theorem to prove it.

Let $\Omega$ be a $C^{2,\alpha}$ bounded domain in $\mathbb{R}^2$ and $g\in C^{2,\alpha}(\partial\Omega) $ with $\alpha\in (0,1 )$. We first have a $C^{0}(\Omega)$ estimate.
\begin{lemma}\label{lemma1}
 Let $\Phi$ be the $C^{2,\alpha}$solution of Dirichlet boundary value problem
\begin{equation}\label{Max}
\left\{
\begin{aligned}
\Delta \Phi &=\varepsilon^{-2}(1-\left(x^{2}+y^{2}\right))  &\text { in } \Omega,\\
\Phi&=g  &\text { on } \partial \Omega.
\end{aligned}
\right.
\end{equation}
Then a solution u to (\ref{fu2}) satisfies
\begin{equation}\label{my}
\sup _{\bar\Omega} |u |\leqslant \sup_{\bar\Omega }2|\Phi|.
\end{equation}

\end{lemma}
\begin{proof}
The existence of $\Phi$ is from Green formula (one can see \cite{HL} ). Consider $u-\Phi$ for $u>0$. Note that on $\mathcal{O}=\{x \in \Omega:$ $u(x)>0\}$
$$\Delta(u-\Phi) =\varepsilon^{-2}(e^{u}-1+\left(x^{2}+y^{2}\right)\left(1-e^{-u}\right))>0.$$
By maximum principle,  we have

$$
\sup _{\mathcal{O}}(u-\Phi)=\sup _{\partial \mathcal{O}}(u-\Phi) \leqslant \sup _{\Omega}\{-\Phi, 0\}
\leqslant\sup_\Omega |\Phi|.$$
It follows
\begin{equation}\label{up}
\sup _{\Omega} u \leqslant \sup_\Omega 2|\Phi|.
\end{equation}
 Similarly, if $u<0$,
$$ \Delta(u-\Phi) =\varepsilon^{-2}(e^{u}-1+\left(x^{2}+y^{2}\right)\left(1-e^{-u}\right))<0.$$
Hence we obtain on $\mathcal{O}^{-}=\{x \in \Omega: u(x)<0\}$
$$\sup _{\mathcal{O}^{-}}(\Phi-u)=\sup _{\partial \mathcal{O}^{-}}(\Phi-u) \leqslant\sup _{\Omega} \{\Phi ,0\}\leqslant \sup_\Omega |\Phi|$$
which implies \begin{equation}\label{um}\sup _{\Omega} {-u} \leqslant \sup_\Omega 2|\Phi|.\end{equation}
Therefore from (\ref{up}) and (\ref{um}) one can get the estimate   (\ref{my}).
\end{proof}

Second, we give the gradient estimate of $u$.

\begin{lemma}\label{lemma2}
Suppose $u\in C^{2}(\Omega)$ satisfies the equation (\ref{fu2}) in $\Omega$, then there is positive constant $C$ depend only on $\Omega$ and $g$ such that
\begin{equation}\label{ugs}
\sup_{\bar \Omega}|\nabla u|\leqslant C.
\end{equation}

\end{lemma}

\begin{proof}
From  the equation (\ref{fu2})  and by standard regularity, one can see $u$ is $C^4$  since $u$ and $g$ are $C^2$.   Then we have
\begin{align}\label{gest}
\Delta|\nabla u|^2&=<\nabla \Delta u, \nabla u>+|\nabla^2u|^2\\ \nonumber
&=\varepsilon^{-2}(e^u+r^2e^{-u})|\nabla u|^2-\varepsilon^{-2}e^{-u}<\nabla r^2,\nabla u>+|\nabla^2 u|^2.
\end{align}

If $|\nabla u|^2$ attains its maximum on the boundary $\partial \Omega$,  we have $\sup |\nabla u|=\sup |\nabla g|$ which leads to (\ref{ugs}).  Now assume $|\nabla u|^2$ attains its maximum at $z_0\in \Omega$. Then from (\ref{gest}), at the point $z_0$  we have
$$\varepsilon^{-2}(e^u+r^2e^{-u})|\nabla u|^2-\varepsilon^{-2}e^{-u}<\nabla r^2,\nabla u>\leqslant 0$$
or
$$(e^u+r^2e^{-u})|\nabla u|^2\leqslant e^{-u} |\nabla r^2| |\nabla u|.$$
Since $|u|$ is bounded from Lemma \ref{lemma1},  there is a constant $C$ dependent on $\Omega$ and $g$ such that
$$ |\nabla u|(z_0)\leqslant C. $$
Then we finish the proof.

\end{proof}

Now we give the proof of Theorem \ref{ThmE}.

\begin{proof}
Let $\sigma\in [0, 1]$, we claim if $u_\sigma\in C^{2,\alpha}(\Omega)$ is the solution of  boundary value problem
\begin{equation}\label{pfs}
\left\{
\begin{aligned}
\Delta u &=\sigma\varepsilon^{-2}\left(e^{u}-\left(x^{2}+y^{2}\right) e^{-u}\right) &\text { in } \Omega,\\
u&=\sigma g  &\text { on } \partial \Omega,
\end{aligned}
\right.
\end{equation}
then there is a constant $M$ independent of $u_\sigma$ and $\sigma$ such that
\begin{equation}\label{claim}
||u_\sigma||_{C^{1,\alpha}(\bar{\Omega})}\leqslant M.
\end{equation}
Then one can use the Leray-Schauder existence theorem (see Theorem 6.23 in \cite{HL})  to show the Dirichlet problem (\ref{fu2}) is solvable in $C^{2,\alpha}(\bar{\Omega})$.

In fact,  one can see $\sigma\Phi$ solves
\begin{equation}\label{phs}
\left\{
\begin{aligned}
\Delta \Phi&=\sigma\varepsilon^{-2} (1-(x^{2}+y^{2}))&\text { in } \Omega,\\
\Phi&=\sigma g  &\text { on } \partial \Omega.
\end{aligned}
\right.
\end{equation}
Then from Lemma \ref{lemma1} , we have
\begin{equation}\label{phs2}
||u_\sigma||_{C^{0}(\bar{\Omega})}\leqslant \sup_\Omega 2|\sigma\Phi |\leqslant \sup_\Omega 2|\Phi|.
\end{equation}

Therefore, from (\ref{pfs}), there is a constant C  independent  on $\sigma$ and $u_\sigma$, such that
 $$|\Delta u|\leqslant C.$$
 This means $|\nabla^2 u|$ is also bounded. From Lemma \ref{lemma2} and  using interpolation inequality in H\"older space, there is a constant $M$ independent on $u$ and $\sigma$ such that (\ref{claim}) is satisfied.

 In the end, by standard bootstrap argument of the regularity we have $u$ is smooth if $\Omega$ and   $g$ are smooth. This finishes the proof.

\end{proof}

\section{Radial symmetry of the global $C^{2}$ solution of the equation in $\mathbb{R}^{2}$}
In this section we will prove Theorem \ref{ThmS}.  In \cite{GNN} and \cite{C}, the radially symmetry of the $C^{2}$ positive solutions of the following second order elliptic equation is studied$$\Delta u+f(u)=0  \text { in } \mathbb{R}^{n} $$
under the assumption on $f$ and $u$. For example, they assumed $u(x)\to 0 $ as $x\to \infty$. Obviously, our equation (\ref{fu2}) is different from this type since $e^{u}-r^2e^{-u}$ has the term $r^2$.  Also we cannot assume $|u|\to 0$ as $r\to +\infty$. In fact, it will lead  $\Delta u\to -\infty$  and then  $u$ is unbounded.  It contradicts the hypothesis.  In this paper,  we assume  for any finite constant $c$
 \begin{equation}\label{asp}
 u(s)-u(t) \rightarrow 0\quad \text{and}\quad |s|-|t|=c, \quad\text{ as}\quad |s|,|t| \rightarrow \infty
 \end{equation}
where $s,t\in \mathbb{R} ^2$.

\begin{proof} \textbf{Proof of Theorem\ref{ThmS}}
Since the partial differential equation (\ref{fu3}) is rotationally symmetric, we only have to prove the symmetry about a line across origin. Here we choose the line $y$ axis.
Define $$\Sigma(\lambda)=\left\{\left(x, y\right) \in \mathbb{R}^{2} \mid x<\lambda\right\}$$ and let $$v=u(2\lambda-x,y),\ x^{\lambda}=2\lambda-x.$$

In $\Sigma(\lambda)$ we define $$w=v(x)-u(x).$$
When $\lambda = 0$ and $x\in \Sigma(\lambda)$,  we have  $x+x^\lambda = 0$ and $x<x^\lambda$. Then $$x^2= (x^\lambda)^2$$
and
\begin{equation}\label{b}
\Delta v -\varepsilon^{-2}\left(e^{v}-\left(({x^{\lambda}})^{2}+y^{2}\right) e^{-v}\right)=\Delta v -\varepsilon^{-2}\left(e^{v}-\left({x}^{2}+y^{2}\right) e^{-v}\right)= 0.
\end{equation}
By the mean value theorem, we have
$$\Delta w+\bar{c} w =  0$$
 where
$$\bar{c}=-\int_{0}^{1}  \varepsilon ^{-2} (e^{ u+t w}+r^2e^{-u-wt} )d t<0.$$
   Then from the assumption (\ref{asp}) and $w(0,0)=0$ on the $y$ axis, we have by maximum principle and minimum principle
$$ w=0$$
in $\Sigma (0)$.  That's to say the global solution of (\ref{fu3}) in $\mathbb{R}^{2}$ is symmetric about $y$ axis.

In the end,  assuming $\lambda > 0$ and $x\in \Sigma(\lambda)$,  then we have   $x+x^\lambda > 0$ and $x<x^\lambda$.  It follows $$x^2< (x^\lambda)^2$$
which implies
\begin{equation}\label{d1}
\Delta v -\varepsilon^{-2}\left(e^{v}-\left({x}^{2}+y^{2}\right) e^{-v}\right)< 0.
\end{equation}
Then by mean value theorem, in $\Sigma(\lambda)$
$$\Delta w+\bar{c} w <  0$$
with $c<0$. Using the infinite boundary condition (\ref{asp}) and $w(\lambda,\lambda)=0$, we have by maximum principle, in $\Sigma(\lambda)$
$$w\geqslant 0.$$

Then if $x>0$ and let $x_\lambda \to x$,  we have $\frac{\partial u}{\partial x}\geqslant 0$. Since $u$ is radially symmetric and  from
$$\frac{\partial u}{\partial x}=\frac{\partial u}{ \partial r} \frac{\partial r}{\partial x}=\frac{\partial u}{\partial r} \frac{x}{r}$$
it follows $\frac{\partial u}{ \partial r} \geqslant0 $ and we finish the proof.

\end{proof}

\end{document}